\documentclass[11pt]{article}

\usepackage{amsmath,amssymb,amsthm,mathtools}
\usepackage[margin=1.05in]{geometry}
\usepackage[T1]{fontenc}
\usepackage{lmodern}
\usepackage{enumitem}

\DeclareSymbolFont{sfoperators}{OT1}{cmss}{m}{n}

\makeatletter
\renewcommand{\operator@font}{\mathgroup\symsfoperators}
\makeatother

\newcommand{\R}{\mathbb R}
\newcommand{\C}{\mathbb C}
\newcommand{\calH}{\mathcal H}
\newcommand{\calV}{\mathcal V}
\newcommand{\calL}{\mathcal L}
\newcommand{\calA}{\mathcal A}

\newcommand{\SDiff}{\operatorname{SDiff}}
\newcommand{\Diff}{\operatorname{Diff}}
\newcommand{\diver}{\operatorname{div}}
\newcommand{\Span}{\operatorname{span}}
\newcommand{\Lie}{\operatorname{Lie}}
\newcommand{\ad}{\operatorname{ad}}
\newcommand{\calM}{\mathcal{M}}

\newtheorem{theorem}{Theorem}
\newtheorem{proposition}[theorem]{Proposition}
\newtheorem{lemma}[theorem]{Lemma}

\theoremstyle{remark}

\title{Rotations, Shears, and Two-Generation on Spheres}
\author{Dennis DeTurck\\[3pt]\small Department of Mathematics\\\small University of Pennsylvania}
\date{\today}

\begin{document}
\maketitle

\begin{abstract}
For every round sphere $S^n$, $n\ge2$, we exhibit two explicit smooth
vector fields whose flows generate a $C^\infty$-dense subgroup of
$\Diff_0(S^n)$.  By deleting one gradient component from the second
generator, we obtain two explicit divergence-free vector fields whose flows
generate a $C^\infty$-dense subgroup of $\SDiff_0(S^n)$.  In both cases
two generators are minimal.

The proof first generates the divergence-free vector fields using rotations,
a linearly modulated rotational shear, and transverse vector spherical
harmonics.  On $S^2$ this reduces to the Hamiltonian/Poisson formalism; for
$n\ge4$ irreducibility of the transverse harmonic spaces completes the
degree-raising argument, while $S^3$ requires a separate treatment of its
two curl eigenspaces.  Adding one gradient field and using the divergence map
then generates all smooth vector fields.  Finally, a Casimir operator for a
recovered copy of $\mathfrak{so}(3)$ separates the components of the
generating fields and compresses the natural finite generating families to
two.
\end{abstract}

\section{Introduction}

A rotational shear of a sphere is obtained by choosing an axis and a
rotation plane perpendicular to it, and rotating that plane through an angle
that depends on the coordinate along the axis.  Such maps arose naturally
in our work on the generalized Minkowski problem
 \cite{DeTurckMinkowski}, where a restricted family of
volume-preserving diffeomorphisms suffices to solve a first-moment centering
problem.  This led first to the question of how few volume-preserving flows
are needed to generate, up to smooth approximation, the entire identity
component of the volume-preserving diffeomorphism group of a sphere.  The
same construction leads naturally to the corresponding question for the full
diffeomorphism group.

Let
$$
  S^n=\{x=(x_1,\ldots,x_{n+1})\in\R^{n+1}:|x|=1\},
  \qquad n\ge2,
$$
and for $i\ne j$ put
$$
  R_{ij}
  =
  x_j\frac{\partial}{\partial x_i}
  -
  x_i\frac{\partial}{\partial x_j}.
$$
The flow of $R_{ij}$ is ordinary rotation in the $x_ix_j$-plane.  We also
use the linearly modulated rotational shear
$$
  S=x_3R_{12}.
$$
Its flow is
$$
  (x_1+ix_2,x_3,x_4,\ldots,x_{n+1})
  \longmapsto
  \bigl(e^{-itx_3}(x_1+ix_2),x_3,x_4,\ldots,x_{n+1}\bigr).
$$
Both the rotations and $S$ are divergence-free.

For the full diffeomorphism group we add one gradient component.  Define
$$
  h_n=
  \begin{cases}
    x_3, & n=2,\\[2mm]
    2x_3^3-3x_3(x_1^2+x_2^2), & n\ge3,
  \end{cases}
  \qquad
  D_n=\nabla_{S^n}h_n.
$$
In either case $h_n$ is a scalar spherical harmonic depending only on the
first three coordinates and invariant under rotation in the $x_1x_2$-plane.

The two divergence-free vector fields used below are
$$
  A=R_{12},
  \qquad
  B=S+R_{23}+R_{34}+\cdots+R_{n,n+1},
$$
and for the full diffeomorphism group we replace $B$ by
$$
  \widetilde B=B+D_n.
$$
For vector fields $X,Y$, let $G(X,Y)$ denote the subgroup generated by the
two one-parameter groups $\{\exp(tX)\}$ and $\{\exp(tY)\}$, $t\in\R$.
Thus
$$
  G(A,B)\subset\SDiff_0(S^n),
  \qquad
  G(A,\widetilde B)\subset\Diff_0(S^n).
$$
Here $\SDiff_0(S^n)$ denotes the group of smooth volume-preserving
diffeomorphisms smoothly isotopic to the identity through
volume-preserving diffeomorphisms, and $\Diff_0(S^n)$ denotes the identity
component of the full smooth diffeomorphism group.

Our main result is the following.

\begin{theorem}\label{thm:main}
For every $n\ge2$:
\vspace{-5pt}
\begin{enumerate}[label=\textup{(\alph*)}]
\item
$
  \overline{G(A,\widetilde B)}^{\,C^\infty}
  =
  \Diff_0(S^n).
$
\item
$
  \overline{G(A,B)}^{\,C^\infty}
  =
  \SDiff_0(S^n).
$
\end{enumerate}
In particular, each of the groups $\Diff_0(S^n)$ and $\SDiff_0(S^n)$ admits a $C^\infty$-dense subgroup 
generated by two explicit one-parameter groups. 
\end{theorem}

The number of generators in Theorem~\ref{thm:main} is minimal.  Indeed, a subgroup generated by a 
single one-parameter group is abelian, and so is its closure, whereas both $\Diff_0(S^n)$ and $\SDiff_0(S^n)$ are nonabelian.

The dimension-dependent definition of $D_n$ is deliberate.  On $S^2$ the
linear gradient $\nabla x_3$ can be separated from the shear by the Casimir
operator of the recovered rotation algebra.  In dimensions $n\ge3$ a new
rotational term $R_{34}$ appears with the same Casimir eigenvalue $-2$ as
$\nabla x_3$.  A quadratic gradient has eigenvalue $-6$, which coincides with
the eigenvalue of the shear $S$.  The cubic harmonic above has eigenvalue
$-12$, the first new value, and therefore permits the required separation.
This is the only reason for using a cubic rather than the simpler linear
gradient that suffices on $S^2$.

The proof of Theorem~\ref{thm:main} is most transparent if one first allows
a somewhat larger, geometrically natural generating family.  Set
$$
  \mathcal E_n
  =
  \{R_{12},R_{23},\ldots,R_{n,n+1},S\}
\qquad\mbox{and}\qquad
  \widetilde{\mathcal E}_n
  =
  \mathcal E_n\cup\{D_n\}.
$$
We first prove the following intermediate result.

\begin{proposition}\label{prop:nplusone}
For every $n\ge2$:
\vspace{-5pt}
\begin{enumerate}[label=\textup{(\alph*)}]
\item
the flows of the vector fields in $\widetilde{\mathcal E}_n$ generate a
$C^\infty$-dense subgroup of $\Diff_0(S^n)$;
\item
the flows of the vector fields in $\mathcal E_n$ generate a
$C^\infty$-dense subgroup of $\SDiff_0(S^n)$.
\end{enumerate}
\end{proposition}

Part~\textup{(b)} is the harmonic-analytic core of the argument.  The
adjacent rotation fields generate the full rotation algebra
$\mathfrak{so}(n+1)$, while $S=x_3R_{12}$ is a degree-two transverse vector
spherical harmonic.  The Lie algebra generated by these fields contains one
nonzero transverse harmonic in each successive degree.  For $n=2$ and
$n\ge4$, irreducibility under $SO(n+1)$ then supplies the entire harmonic
space, while for $n=3$ an additional argument is required.  Spectral
completeness shows that finite sums of these harmonics are
$C^\infty$-dense among all smooth divergence-free vector fields.

The mechanism that raises the harmonic degree is essentially
three-dimensional.  A fixed coordinate subspace
$$
  \Span\{e_1,e_2,e_3\}\subset\R^{n+1}
$$
carries the same calculation that appears on $S^2$, and that calculation
survives inside the transverse vector harmonics in every dimension.  On
$S^2$ the picture is especially simple: every divergence-free vector field
is Hamiltonian, and the relevant bracket calculation becomes a Poisson
identity between scalar spherical harmonics.  The three-sphere is the only
representation-theoretic exception, reflecting the familiar decomposition
$$
  \mathfrak{so}(4)
  \cong
  \mathfrak{so}(3)_+\oplus\mathfrak{so}(3)_-.
$$
There each transverse harmonic space splits into two inequivalent
irreducible summands, and one must show that both summands occur.

Part~\textup{(a)} builds on part~\textup{(b)}.  The divergence-free fields
form the kernel in the exact sequence
$$
  0
  \longrightarrow
  \mathfrak X_{\diver=0}(S^n)
  \longrightarrow
  \mathfrak X(S^n)
  \overset{\diver}{\longrightarrow}
  C^\infty_0(S^n)
  \longrightarrow
  0,
$$
where $C^\infty_0(S^n)$ denotes the mean-zero smooth functions.  If $X$ is
divergence-free, then
$$
  \diver[X,Y]=X(\diver Y),
$$
so the already-generated divergence-free algebra acts on the divergences of
the enlarged Lie algebra.  For $n=2$,
$$
  \diver D_2=\Delta_{S^2}x_3=-2x_3,
$$
and the degree-raising calculation generates every nonconstant scalar
spherical harmonic from degree one.  For $n\ge3$, the cubic seed begins in
degree three; an explicit embedded $S^2$ calculation produces degrees two
and one, after which the same raising mechanism applies.  Thus the
divergences are dense in $C^\infty_0(S^n)$, and the divergence-free kernel
together with these quotient directions gives density in all smooth vector
fields.

The second step is finite-dimensional in spirit.  Although
$\mathcal E_n$ and $\widetilde{\mathcal E}_n$ contain several geometrically
natural fields, each family can be recovered from the corresponding pair in
Theorem~\ref{thm:main}.

\begin{proposition}\label{prop:two-recover}
For every $n\ge2$:
\vspace{-5pt}
\begin{enumerate}[label=\textup{(\alph*)}]
\item
the Lie algebra generated by $A$ and $\widetilde B$ contains every field in
$\widetilde{\mathcal E}_n$;
\item
the Lie algebra generated by $A$ and $B$ contains every field in
$\mathcal E_n$.
\end{enumerate}
\end{proposition}

In both cases the first brackets recover the rotations in the first three
coordinates.  Let
$$
  \mathcal C
  =
  \ad_{R_{12}}^2+\ad_{R_{13}}^2+\ad_{R_{23}}^2
$$
be the Casimir operator for this copy of $\mathfrak{so}(3)$.  For $n=2$,
after subtracting the recovered $R_{23}$ from $\widetilde B$, the two
remaining components $D_2=\nabla x_3$ and $S$ have Casimir eigenvalues $-2$
and $-6$.  For $n\ge3$, the components of
$$
  \widetilde B-R_{23}
  =
  D_n+S+R_{34}+\Bigl(R_{45}+\cdots+R_{n,n+1}\Bigr)
$$
have eigenvalues
$-12$, $-6$, $-2$, and 0,
respectively.  Polynomial filtering therefore recovers $D_n$ and $S$.
Once $S$ has been extracted, the adjacent rotations can be recovered
successively.  Part~\textup{(b)} is the same argument with the gradient
component omitted; the relevant eigenvalues are then $-6,-2,0$.

Propositions~\ref{prop:nplusone} and \ref{prop:two-recover} nearly prove
Theorem~\ref{thm:main} by themselves.  One additional standard observation
is needed: if a vector field belongs to the Lie algebra generated by a
finite collection of fields, then for every finite $k$ its flow belongs to
the $C^k$ closure of the subgroup generated by their flows.  This follows
from Lie--Trotter and commutator product formulas.  Time slicing then passes
from autonomous to time-dependent vector fields, and the final passage from
finite $C^k$ approximation to the $C^\infty$ topology is made at the end of
the proof.  Applied to Proposition~\ref{prop:nplusone}, this yields both
parts of Theorem~\ref{thm:main}.

There is substantial related work on finite generation and approximate
controllability by flows.  Agrachev and Sarychev developed the
Lie-algebraic approximation mechanism in diffeomorphism groups and, on
$S^2$, obtained finite generating families expressed in spherical
harmonics \cite{AgrachevSarychev2008,AgrachevSarychev2022}.  Very recently,
Jiang and Chen proved that the Lie algebra of all algebraic vector fields on
$S^2$ is generated by two explicit vector fields, and consequently obtained
two-flow approximation of $\Diff_0(S^2)$ \cite{JiangChen2026}.  Thus they
have priority for the two-generator theorem for the full diffeomorphism
group of $S^2$; their result is also stronger at the algebraic level, since
their two fields generate the algebraic vector-field Lie algebra exactly.
They announce a forthcoming extension of their algebraic result to spheres
in arbitrary dimension.

The constructions are closely related on $S^2$.  In both, one generator is
a rotation and the other packages together Hamiltonian and gradient
components that can be separated spectrally.  The present proof is organized
so that the divergence-free kernel is generated first and the remaining
directions are then recovered through the divergence quotient.  This
organization persists in every dimension and simultaneously yields the
volume-preserving statement: deleting the gradient component from
$\widetilde B$ gives the divergence-free pair $A,B$ in
Theorem~\ref{thm:main}\textup{(b)}.

The remainder of the paper follows this structure.  We first develop the
transverse harmonic decomposition and prove the divergence-free part of
Proposition~\ref{prop:nplusone}, beginning with the Hamiltonian argument on
$S^2$, then embedding that mechanism into the generic case $n\ge4$, and
finally treating the exceptional case $S^3$.  We next add the gradient
component and use the divergence quotient to obtain all smooth vector fields.
After carrying out the dimension-independent passage from vector fields to
their flows, we prove Proposition~\ref{prop:two-recover} by Casimir
separation and successive recovery of the rotational tail.  This completes
both parts of Theorem~\ref{thm:main}.

\section{Rotations, shears, and vector spherical harmonics}

Theorem~\ref{thm:main} has two parts, but the proof of the
full-diffeomorphism statement begins with the volume-preserving one.  The
divergence-free vector fields form the kernel of the divergence map, and the
fields $R_{ij}$ and $S$ all lie in this kernel.  We therefore begin with
the Lie algebra
$$
  \calL_n
  =
  \Lie\{R_{12},R_{23},\ldots,R_{n,n+1},S\},
  \qquad
  S=x_3R_{12},
$$
where $\Lie$ denotes the Lie algebra of smooth vector fields generated by
the indicated fields.  In this and the next three sections we show that
$\calL_n$ is $C^\infty$-dense in the smooth divergence-free vector
fields on $S^n$.  This proves the infinitesimal part of
Proposition~\ref{prop:nplusone}\textup{(b)} and supplies the kernel needed for
part~\textup{(a)}.  The passage from vector fields to their flows will be
carried out afterward.

Only then do we add the gradient field $D_n$.  If $X$ is
divergence-free, then
$$
  \diver[X,Y]
  =
  X(\diver Y),
$$
so the already-generated divergence-free algebra acts on the divergences of
the enlarged Lie algebra.  The remaining problem can therefore be reduced to
generation of scalar spherical harmonics in the quotient
$C^\infty_0(S^n)$.

We will repeatedly regard homogeneous polynomial vector fields on $\mathbb R^{n+1}$ as vector fields on $S^n$ by restriction. For such a field $P$ that is tangent to the sphere, its ambient and intrinsic divergences agree on $S^n$. Indeed, with $\nu=x$ the outward unit normal,
$$
\operatorname{div}_{\mathbb R^{n+1}}P
=
\operatorname{div}_{S^n}(P|_{S^n})
+
\langle D_\nu P,\nu\rangle .
$$
If $P$ is homogeneous of degree $d$, Euler's identity gives $D_\nu P=dP$ on $S^n$, and tangency implies
$$
\langle D_\nu P,\nu\rangle=d\langle P,\nu\rangle=0.
$$
Thus, for the homogeneous tangent fields considered below, ambient and intrinsic divergence coincide on the sphere; we will therefore simply write $\diver$, rather than $\diver_{\mathbb R^{n+1}}$ or $\diver_{S^n}$.

The ambient field $R_{ij}$ is tangent to every sphere, Killing, and
divergence-free.  The ``adjacent'' fields
$$
  R_{12},R_{23},\ldots,R_{n,n+1}
$$
generate the full Lie algebra $\mathfrak{so}(n+1)$, since their iterated
brackets produce every $R_{ij}$.  Consequently
$
  \mathfrak{so}(n+1)\subset\calL_n.
$

The shear
$
  S=x_3R_{12}
$
has ambient expression
$$
  S(x)
  =
  (x_2x_3,-x_1x_3,0,\ldots,0).
$$
Hence
$$
  \langle x,S(x)\rangle=0,
  \qquad
  \diver S=0,
  \qquad
  \Delta S=0
$$
componentwise.  Thus $S$ is tangent, divergence-free, and homogeneous
harmonic of degree two.

For $d\ge1$, let $\calV_d^{(n)}$ denote the space of restrictions to $S^n$
of homogeneous polynomial vector fields
$
  P:\R^{n+1}\longrightarrow\R^{n+1}
$
of degree $d$ satisfying
$$
  \Delta P=0,
  \qquad
  \diver P=0,
  \qquad
  \langle x,P(x)\rangle=0.
$$
The last condition says that $P$ is tangent to the sphere.  

The vector fields in $\calV_d^{(n)}$ are the transverse vector spherical harmonics of degree $d$.
The
first space is
$$
  \calV_1^{(n)}
  =
  \{Ax:A\in\mathfrak{so}(n+1)\},
$$
so degree one consists precisely of the infinitesimal rotations.  In
particular,
$
  \calV_1^{(n)}\subset\calL_n,
$
while the shear $S$ is a nonzero element of $\calV_2^{(n)}$.

Under the metric identification of vector fields with one-forms, a tangent
vector field is divergence-free precisely when the corresponding one-form is
coclosed.  Since $H^1(S^n;\R)=0$ for $n\ge2$, the coclosed one-forms are
exactly the coexact one-forms.  The Hodge Laplacian preserves this subspace.
The classical calculation of its spectrum on the round sphere identifies
$\calV_d^{(n)}$ with the coexact one-form eigenspace having eigenvalue
$$
  \lambda_d=(d+1)(d+n-2),
  \qquad d=1,2,\ldots;
$$
see \cite{IkedaTaniguchi,RubinOrdonez}.

We record explicitly the completeness statement that will be used
throughout the proof.

\begin{lemma}[Spectral completeness]\label{lem:spectral-completeness}
Every smooth divergence-free vector field $X$ on $S^n$ has an expansion
$$
  X=\sum_{d=1}^\infty X_d,
  \qquad X_d\in\calV_d^{(n)},
$$
which converges in $C^k$ for every finite $k$, and hence in the
$C^\infty$ topology.  In particular, finite sums of transverse vector
spherical harmonics are $C^\infty$-dense in the smooth divergence-free
vector fields.
\end{lemma}

\begin{proof}
The Hodge Laplacian on one-forms on the compact sphere is elliptic and
self-adjoint, so its eigenforms form a complete orthonormal basis of $L^2$.
Hodge decomposition and $H^1(S^n;\R)=0$ identify the divergence-free
sector with the coexact sector, whose eigenspaces are precisely the
$\calV_d^{(n)}$ described above.

Let $\alpha$ be the one-form corresponding to a smooth divergence-free
field and write $\alpha=\sum_d\alpha_d$ for its coexact spectral
expansion.  For every integer $N\ge0$, smoothness of $\alpha$ implies
$(1+\Delta)^N\alpha\in L^2$, and therefore
$$
  \sum_d(1+\lambda_d)^{2N}\|\alpha_d\|_{L^2}^2<\infty.
$$
It follows that the spectral tails tend to zero in every Sobolev norm.
Sobolev embedding then gives convergence in every $C^k$ norm.  Since this
holds for all finite $k$, the convergence is $C^\infty$.
\end{proof}

The representation theory has one important dimension-dependent feature.
For $n\ge4$, $\calV_d^{(n)}$ is an irreducible $SO(n+1)$-module; in highest-weight, or equivalently 
Young-diagram, language it is the trace-free hook representation of type $(d,1)$; see \cite[\text{Theorem 6.8}]{IkedaTaniguchi}.  

On $S^3$, $\calV_d^{(3)}$ splits into two inequivalent irreducible summands, while on $S^2$, $\calV_d^{(2)}$ identifies with the scalar harmonic representation. We begin with $S^2$, where the degree-raising mechanism takes its simplest form in the Hamiltonian/Poisson language. We then embed that calculation into higher-dimensional spheres; for $n\ge4$, irreducibility completes the argument directly, while $S^3$ requires a separate treatment because of the splitting into two curl eigenspaces.


\section{The two-sphere: Hamiltonians and Poisson brackets}

We begin with $S^2$, where the degree-raising mechanism underlying the proof takes its 
simplest form. Let $\omega$ be the standard area form on $S^2$.
If $Y$ is divergence-free, then
$$
  0=\mathcal L_Y\omega=d(\iota_Y\omega).
$$
Since
$
  H^1(S^2;\R)=0,
$
there is a smooth function $f$, unique up to an additive constant, such that
$
  \iota_Y\omega=-df.
$
Thus every divergence-free vector field is Hamiltonian:
$
  Y=X_f.
$

In this section, we write $(x,y,z)=(x_1,x_2,x_3)$. For a smooth extension of $f$ to $\R^3$,
$
  X_f(p)=p\times\nabla f(p).
$
The corresponding Poisson bracket is
$$
  \{f,g\}
  =
  X_f(g)
  =
  \langle p,\nabla f\times\nabla g\rangle,
$$
and
$
  [X_f,X_g]=X_{\{f,g\}}.
$

In particular,
$$
  \{x,y\}=z,
  \qquad
  \{y,z\}=x,
  \qquad
  \{z,x\}=y.
$$

Let $\calH_d$ denote the space of scalar spherical harmonics of degree $d$.
Degree by degree, the Hamiltonian correspondence gives an
$SO(3)$-equivariant isomorphism
$$
  \calH_d
  \longrightarrow
  \calV_d^{(2)},
  \qquad
  f\longmapsto X_f.
$$
Thus on $S^2$ the transverse vector-harmonic decomposition is equivalent
to the usual scalar spherical-harmonic decomposition.

For $n=2$ our generating family is
$$
\mathcal E_2=\{R_{12},R_{23},S\},
$$
and hence
$$
\calL_2=\Lie\{R_{12},R_{23},S\}.
$$
The rotation fields $R_{12}$ and $R_{23}$ have Hamiltonians $z$ and $x$,
respectively.  Their bracket gives the rotation with Hamiltonian $y$, so
$
  \calV_1^{(2)}\subset\calL_2.
$
The shear $S=zR_{12}$ has Hamiltonian $z^2/2$. On the sphere this differs by a 
constant from one-half of the nonzero quadratic harmonic
$$
z^2-\frac13
=
z^2-\frac13(x^2+y^2+z^2).
$$
Thus $\calL_2\cap\calV_2^{(2)}$ is nonzero.
Since $\calL_2$ contains the infinitesimal rotations, 
this intersection is an $\mathfrak{so}(3)$-submodule of $\calV_2^{(2)}$. Irreducibility of $\calH_2$, or 
equivalently of $\calV_2^{(2)}$, therefore gives
$
\calV_2^{(2)}\subset\calL_2.
$

The key degree-raising calculation is equally simple. Put $u=x+iy$, and write
$$
u^d=a_d+ib_d,
\qquad
a_d=\Re u^d,\quad b_d=\Im u^d.
$$
Since $\{z,u\}=-iu$,
$$
\left\{\frac12z^2,u^d\right\}=-id\,zu^d.
$$
Equating real and imaginary parts gives
$$
\left\{\frac12z^2,a_d\right\}=d\,zb_d,
\qquad
\left\{\frac12z^2,b_d\right\}=-d\,za_d.
$$
Moreover, $zu^d$ is homogeneous harmonic of degree $d+1$, since
$$
\Delta(zu^d)
=
z\Delta(u^d)+2\frac{\partial u^d}{\partial z}
=0.
$$
Thus $za_d$ and $zb_d$ are real spherical harmonics of degree $d+1$.

\begin{proposition}\label{prop:two-sphere}
For every $d\ge1$,
$
  \calV_d^{(2)}\subset\calL_2.
$
\end{proposition}

\begin{proof}
We have already established the result in degrees one and two.  
Suppose inductively that $\calV_d^{(2)}\subset\calL_2$. The Hamiltonian fields $X_{a_d}$ and $X_{b_d}$ therefore belong to $\calL_2$. Since $S=X_{z^2/2}$, the identities above give
$$
[S,X_{a_d}]=dX_{zb_d},
\qquad
[S,X_{b_d}]=-dX_{za_d}.
$$
In particular, $\calL_2\cap\calV_{d+1}^{(2)}$ is nonzero. Since $\calL_2$ contains the infinitesimal rotations, this intersection is an $\mathfrak{so}(3)$-submodule of $\calV_{d+1}^{(2)}$. Irreducibility therefore gives
$
\calV_{d+1}^{(2)}\subset\calL_2.
$
\end{proof}

Proposition~\ref{prop:two-sphere} and
Lemma~\ref{lem:spectral-completeness} therefore imply that $\calL_2$ is
$C^\infty$-dense in the space of smooth divergence-free vector fields on
$S^2$.

In this Hamiltonian picture, the motivation for using the linearly modulated
shear becomes particularly transparent: it is simply the Hamiltonian flow
associated with the elementary quadratic function $z^2/2$, the simplest
nonlinear Hamiltonian beyond the rotational ones.

The calculation above is not peculiar to $S^2$.  In the next section we embed it into the first three coordinate 
directions of $S^n$; for $n\ge4$, irreducibility of the transverse harmonic spaces then turns this three-dimensional 
degree-raising mechanism into the full generation result.

\section{The generic case: $n\ge4$}

The degree-raising calculation on $S^2$ can be embedded directly into
every higher-dimensional sphere.  Fix the three-dimensional coordinate
subspace
$
  U=\Span\{e_1,e_2,e_3\}\subset\R^{n+1}.
$

We will use $U\cap S^n$ as the coordinate two-sphere on which the Hamiltonian 
calculation takes place.  Write
$
  x_U=(x_1,x_2,x_3).
$
If $f(x_U)=f(x_1,x_2,x_3)$ is a homogeneous polynomial on
$U\cong\R^3$, we continue to write $X_f$ for the corresponding
three-dimensional Hamiltonian field, extended trivially in the remaining
coordinates:
$$
  X_f(x)
  =
  \bigl(x_U\times\nabla_U f(x_U),0,\ldots,0\bigr).
$$

\begin{lemma}\label{lem:embedded}
If $f$ is homogeneous harmonic of degree $d$ on $\R^3$, then
$
  X_f\in\calV_d^{(n)}.
$

For arbitrary smooth functions $f$ and $g$ of $x_U$,
$$
  [X_f,X_g]
  =
  X_{\{f,g\}},
$$
where
$$
  \{f,g\}
  =
  \langle x_U,\nabla_U f\times\nabla_U g\rangle
$$
is the same three-dimensional Poisson bracket used on $S^2$.
\end{lemma}

\begin{proof}
Suppose first that $f$ is homogeneous harmonic of degree $d$.  Tangency
follows from
$$
  \langle x,X_f(x)\rangle
  =
  \langle x_U,x_U\times\nabla_Uf\rangle
  =0.
$$
The divergence vanishes because
$
  \diver(x_U\times\nabla_Uf)=0.
$
The components of $X_f$ in the first three coordinate directions are
$$
x_2\frac{\partial f}{\partial x_3}
-
x_3\frac{\partial f}{\partial x_2},
\qquad
x_3\frac{\partial f}{\partial x_1}
-
x_1\frac{\partial f}{\partial x_3},
\qquad
x_1\frac{\partial f}{\partial x_2}
-
x_2\frac{\partial f}{\partial x_1}.
$$
Since
$$
\Delta\left(
x_j\frac{\partial f}{\partial x_k}
-
x_k\frac{\partial f}{\partial x_j}
\right)
=
2\frac{\partial^2f}{\partial x_j\partial x_k}
-2\frac{\partial^2f}{\partial x_k\partial x_j} 
+x_j\frac{\partial}{\partial x_k}\Delta f
-x_k\frac{\partial}{\partial x_j}\Delta f
=0,
$$
we have $X_f\in\calV_d^{(n)}$.

The bracket identity does not require harmonicity.  Since both vector fields
act only in the first three coordinates, their Lie bracket is exactly the
Hamiltonian bracket from the preceding $S^2$ calculation, acting in those
coordinates.  Hence
$
  [X_f,X_g]=X_{\{f,g\}}.
$
\end{proof}

The degree-raising identity from $S^2$ therefore survives unchanged in the
embedded copy.  Put
$$
  u=x_1+ix_2,
  \qquad
  z=x_3,
  \qquad
  q=\frac12z^2.
$$
Although $q$ itself is not harmonic, its Hamiltonian field on $U$, extended by 
zero on $U^\perp$, is precisely our linearly modulated shear:
$$
X_q=zR_{12}=S.
$$
As in the two-dimensional case,  write
$$
u^d=a_d+ib_d,
\qquad
a_d=\Re u^d,\quad b_d=\Im u^d.
$$
The polynomials $a_d$ and $b_d$ are homogeneous harmonic polynomials of degree $d$, so Lemma~\ref{lem:embedded} gives
$$
X_{a_d},X_{b_d}\in\calV_d^{(n)}.
$$
The Poisson identities obtained on $S^2$,
$$
\left\{\frac12z^2,a_d\right\}=d\,zb_d,
\qquad
\left\{\frac12z^2,b_d\right\}=-d\,za_d,
$$
therefore give
$$
[S,X_{a_d}]=dX_{zb_d},
\qquad
[S,X_{b_d}]=-dX_{za_d}.
$$
As we saw in the preceding section, $za_d$ and $zb_d$ are homogeneous harmonic polynomials of degree $d+1$. Hence the two vector fields on the right belong to $\calV_{d+1}^{(n)}$.

\begin{proposition}\label{prop:generic}
If $n\ge4$, then
$
  \calV_d^{(n)}\subset\calL_n$
 for every $d\ge1$.
\end{proposition}

\begin{proof}
We already know that
$
\calV_1^{(n)}\subset\calL_n
$
and that the shear $S$ is a nonzero element of $\calV_2^{(n)}$. Since $\calL_n$ contains $\mathfrak{so}(n+1)$, the intersection
$
\calL_n\cap\calV_2^{(n)}
$
is an $\mathfrak{so}(n+1)$-submodule of $\calV_2^{(n)}$. It is nonzero because it contains $S$. For $n\ge4$, $\calV_2^{(n)}$ is irreducible, and therefore
$
\calV_2^{(n)}\subset\calL_n.
$
Suppose inductively that
$
\calV_d^{(n)}\subset\calL_n
$
for some $d\ge2$. Then $X_{a_d}$ and $X_{b_d}$ belong to $\calL_n$. The bracket identities above show that
$$
X_{zb_d},X_{za_d}\in
\calL_n\cap\calV_{d+1}^{(n)}.
$$
In particular, this intersection is a nonzero $\mathfrak{so}(n+1)$-submodule of $\calV_{d+1}^{(n)}$. Irreducibility therefore gives
$
\calV_{d+1}^{(n)}\subset\calL_n.
$
\end{proof}

It follows from Proposition~\ref{prop:generic} and
Lemma~\ref{lem:spectral-completeness} that, for $n\ge4$, the Lie algebra
$\calL_n$ is $C^\infty$-dense in the space of smooth divergence-free
vector fields on $S^n$.

The same embedded three-dimensional calculation also works on $S^3$.
There, however, the transverse harmonic spaces are no longer irreducible,
and a separate argument is needed to show that both irreducible summands
are generated.

\section{The three-sphere}

The exceptional behavior of $S^3$ reflects the exceptional structure of its rotation algebra:
$$
\mathfrak{so}(4)\cong\mathfrak{so}(3)_+\oplus\mathfrak{so}(3)_-.
$$
At degree one this is precisely the decomposition of the infinitesimal rotations into the self-dual and 
anti-self-dual parts. The same distinction persists in the higher transverse harmonics. For each $d\ge1$,
$$
\calV_d^{(3)}=\calV_d^+\oplus\calV_d^-,
$$
where the two summands are inequivalent irreducible $SO(4)$-modules. Under the identification of vector 
fields with one-forms, they are the positive and negative eigenspaces of the curl operator $*d$, with 
eigenvalues $+(d+1)$ and $-(d+1)$, respectively; see
\cite{BarCurl,BenAchourEtAl,IkedaTaniguchi}.

We use the reflection
$
  \rho(x_1,x_2,x_3,x_4)
  =
  (x_1,x_2,x_3,-x_4).
$
This is an orientation-reversing isometry of $S^3$.  If $\alpha$ is a
differential form, then
$
  \rho^*(*\alpha)=-*\rho^*\alpha,
$
while $d$ commutes with pullback.  Hence
$
  \rho^*(*d)=-(*d)\rho^*,
$
so $\rho$ interchanges $\calV_d^+$ and $\calV_d^-$.

On the other hand, every field arising from the embedded
three-dimensional construction is fixed by $\rho$: it depends only on
$x_1,x_2,x_3$ and has no $x_4$-component.  More generally, suppose that
$
  0\ne X\in\calV_d^{(3)}
$
is fixed by $\rho$, and write
$$
  X=X_++X_-,
  \qquad
  X_\pm\in\calV_d^\pm.
$$
Since $\rho$ interchanges the two summands and fixes $X$, it follows that
$$
  \rho_*X_+=X_-,
  \qquad
  \rho_*X_-=X_+.
$$
In particular, both $X_+$ and $X_-$ are nonzero.

The following observation is the key to the exceptional $S^3$ case.

\begin{lemma}\label{lem:S3-reflection}
Let $0\ne X\in\calV_d^{(3)}$ be fixed by $\rho$.  Then the
$\mathfrak{so}(4)$-submodule generated by $X$ is all of
$\calV_d^{(3)}$.
\end{lemma}

\begin{proof}
Let $W$ be the $\mathfrak{so}(4)$-submodule generated by $X$.  The
coordinate projections
$$
  \pi_\pm:W\longrightarrow\calV_d^\pm
$$
are equivariant.  Since the two components $X_+$ and $X_-$ of $X$ are
nonzero, both projections are nonzero.  By irreducibility they are therefore
surjective.

Suppose $\ker\pi_+=0$.  Then $\pi_+$ identifies $W$ equivariantly with
$\calV_d^+$, while $\pi_-$ gives a nonzero equivariant map
$
  \calV_d^+\longrightarrow\calV_d^-,
$
contradicting the inequivalence of the two irreducible summands.  Thus
$\ker\pi_+\ne0$.  But
$
  \ker\pi_+\subset\calV_d^-,
$
and irreducibility of $\calV_d^-$ implies
$
  \calV_d^-\subset W.
$
The same argument with the signs reversed gives
$
  \calV_d^+\subset W.
$
Hence
$$
  W=\calV_d^+\oplus\calV_d^-
   =\calV_d^{(3)}.
$$
\end{proof}

\begin{proposition}\label{prop:S3}
For every $d\ge1$,
$
  \calV_d^{(3)}\subset\calL_3.
$
\end{proposition}

\begin{proof}
Degree one is supplied by the rotations:
$
  \calV_1^{(3)}\subset\calL_3.
$
The shear $S$ is a nonzero element of $\calV_2^{(3)}$ and is fixed by
$\rho$.  Lemma~\ref{lem:S3-reflection} therefore shows that the
$\mathfrak{so}(4)$-submodule generated by $S$ is all of
$\calV_2^{(3)}$.  Since $\calL_3$ contains the infinitesimal rotations and
is closed under Lie brackets,
$
  \calV_2^{(3)}\subset\calL_3.
$

Now suppose inductively that
$
\calV_d^{(3)}\subset\calL_3
$
for some $d\ge2$. 
As in the $S^2$ and $n\ge4$ cases, write $u^d=a_d+ib_d$.
Then $X_{a_d}$ and $X_{b_d}$ belong to $\calL_3$, and bracketing with the shear gives
$$
[S,X_{a_d}]=dX_{zb_d},
\qquad
[S,X_{b_d}]=-dX_{za_d}.
$$
The fields on the right belong to $\calV_{d+1}^{(3)}$, are fixed by $\rho$, and are nonzero. Hence $\calL_3$ contains a nonzero reflection-invariant element of $\calV_{d+1}^{(3)}$. Lemma~\ref{lem:S3-reflection} then implies
$
\calV_{d+1}^{(3)}\subset\calL_3.
$
\end{proof}

Proposition~\ref{prop:S3} and
Lemma~\ref{lem:spectral-completeness} therefore imply that $\calL_3$ is
$C^\infty$-dense in the space of smooth divergence-free vector fields on
$S^3$.

This completes the infinitesimal divergence-free argument in all dimensions.  We now enlarge 
$\calL_n$ by the gradient field $D_n$ and use the divergence map to pass from the kernel to 
all smooth vector fields.

\section{From divergence-free fields to all vector fields}
\label{sec:all-vector-fields}

The preceding four sections show that
$$
  \calL_n
  =
  \Lie\{R_{12},R_{23},\ldots,R_{n,n+1},S\},
  \qquad
  S=x_3R_{12},
$$
is $C^\infty$-dense in the smooth divergence-free vector fields on
$S^n$.  We now add the gradient field $D_n$ introduced in the
Introduction and prove the corresponding density statement for all smooth
vector fields.

Set
$
  \widetilde{\calL}_n
  =
  \Lie\bigl(\calL_n\cup\{D_n\}\bigr).
$
Recall that
$
  D_n=\nabla_{S^n}h_n,
$
where
$$
  h_n=
  \begin{cases}
    x_3, & n=2,\\[2mm]
    2x_3^3-3x_3(x_1^2+x_2^2), & n\ge3.
  \end{cases}
$$

The relevant structure is the exact sequence
$$
  0
  \longrightarrow
  \mathfrak X_{\diver=0}(S^n)
  \longrightarrow
  \mathfrak X(S^n)
  \overset{\diver}{\longrightarrow}
  C^\infty_0(S^n)
  \longrightarrow
  0,
$$
where $C^\infty_0(S^n)$ denotes the smooth mean-zero functions.  The
preceding sections give dense control of the kernel.  We now show that
$\widetilde{\calL}_n$ generates densely in the quotient as well.

Let
$
  \mathcal D_n
  =
  \{\diver Y:Y\in\widetilde{\calL}_n\}.
$
Every element of $\mathcal D_n$ has mean zero.  We shall use repeatedly
the identity
$$
  \diver[X,Y]
  =
  X(\diver Y)-Y(\diver X).
$$
In particular, if $X$ is divergence-free, then
$$
  \diver[X,Y]
  =
  X(\diver Y).
$$
Since $\calL_n\subset\widetilde{\calL}_n$, it follows that
$\mathcal D_n$ is invariant under the action of $\calL_n$.  In
particular, it is invariant under every infinitesimal rotation $R_{ij}$.

Let $\calH_d(S^n)$ denote the scalar spherical harmonics of degree $d$.

\begin{proposition}\label{prop:scalar-divergences}
For every $n\ge2$ and every $d\ge1$,
$
  \calH_d(S^n)\subset\mathcal D_n.
$
\end{proposition}

\begin{proof}
We treat $S^2$ first.  Here
$
  D_2=\nabla_{S^2}x_3,
$
and hence
$$
  \diver D_2
  =
  \Delta_{S^2}x_3
  =
  -2x_3.
$$
Thus $\mathcal D_2$ contains a nonzero element of
$\calH_1(S^2)$.  Since it is invariant under rotations and
$\calH_1(S^2)$ is irreducible,
$
  \calH_1(S^2)\subset\mathcal D_2.
$

The degree-raising calculation from Section~3 now applies directly.  Put
$$
  u=x_1+ix_2,
  \qquad
  u^d=a_d+ib_d.
$$
Suppose
$
  \calH_d(S^2)\subset\mathcal D_2.
$
Then $a_d\in\mathcal D_2$.  Since
$
  S=X_{x_3^2/2}\in\calL_2,
$
invariance of $\mathcal D_2$ gives
$$
  S(a_d)
  =
  \left\{\frac12x_3^2,a_d\right\}
  =
  d\,x_3b_d
  \in
  \mathcal D_2.
$$
The polynomial $x_3b_d$ is a nonzero homogeneous harmonic polynomial of
degree $d+1$.  Rotational irreducibility therefore gives
$
  \calH_{d+1}(S^2)\subset\mathcal D_2.
$
Starting with degree one proves the proposition for $n=2$.

Now suppose $n\ge3$, and write
$$
  h
  =
  2x_3^3-3x_3(x_1^2+x_2^2).
$$
This is a homogeneous harmonic polynomial of degree three in the first
three coordinates, so
$$
  \diver D_n
  =
  \Delta_{S^n}h
  =
  -3(n+2)h.
$$
Hence
$
  h\in\mathcal D_n.
$

We first descend from degree three to degrees two and one.  Put
$
  q=x_1x_3.
$
Its embedded Hamiltonian field $X_q$ belongs to $\calL_n$.  Using the
three-dimensional Poisson bracket from the preceding sections,
$$
  X_q(h)
  =
  \{x_1x_3,h\}
  =
  3x_2x_3
  \bigl(x_1^2+x_2^2-4x_3^2\bigr).
$$
Since $h\in\mathcal D_n$ and $X_q\in\calL_n$, this function belongs to
$\mathcal D_n$.  As an ambient homogeneous polynomial of degree four it
satisfies
$$
  \Delta X_q(h)
  =
  -48x_2x_3,
  \qquad
  \Delta^2X_q(h)=0.
$$
Its restriction to $S^n$ therefore has only degree-four and degree-two
spherical-harmonic components, and the degree-two component is nonzero.

To separate these components, let
$$
  \mathcal K
  =
  \sum_{1\le i<j\le n+1}R_{ij}^2
$$
be the scalar rotational Casimir.  On $\calH_d(S^n)$,
$
  \mathcal K$ is multiplication by $-d(d+n-1)$.
Since $\mathcal D_n$ is invariant under all rotations, it is invariant
under $\mathcal K$.  The degree-four and degree-two eigenvalues are
distinct, so a polynomial in $\mathcal K$ extracts the nonzero degree-two
component.  Irreducibility of $\calH_2(S^n)$ under $SO(n+1)$ then gives
$
  \calH_2(S^n)\subset\mathcal D_n.
$

In particular,
$
  x_2x_3\in\mathcal D_n.
$
Applying $X_q$ once more gives
$$
  X_q(x_2x_3)
  =
  \{x_1x_3,x_2x_3\}
  =
  -x_3(x_1^2+x_2^2-x_3^2).
$$
This homogeneous polynomial of degree three satisfies
$$
  \Delta X_q(x_2x_3)
  =
  2x_3,
  \qquad
  \Delta^2X_q(x_2x_3)=0.
$$
Its restriction to the sphere therefore has only degree-three and degree-one
components, with nonzero degree-one component.  The scalar Casimir separates
the two degrees, and rotational irreducibility gives
$
  \calH_1(S^n)\subset\mathcal D_n.
$

We may now use the same degree-raising argument as in the $S^2$ case.
Suppose
$
  \calH_d(S^n)\subset\mathcal D_n.
$
With
$$
  u=x_1+ix_2,
  \qquad
  u^d=a_d+ib_d,
$$
we have $a_d\in\mathcal D_n$, and hence
$$
  S(a_d)
  =
  \left\{\frac12x_3^2,a_d\right\}
  =
  d\,x_3b_d
  \in
  \mathcal D_n.
$$
The polynomial $x_3b_d$ is a nonzero homogeneous harmonic polynomial of
degree $d+1$.  Rotational irreducibility therefore gives
$
  \calH_{d+1}(S^n)\subset\mathcal D_n.
$
Starting with degree one completes the induction.
\end{proof}

Thus the divergences of fields in $\widetilde{\calL}_n$ contain every
finite scalar spherical-harmonic sum of mean zero.

\begin{proposition}\label{prop:all-vector-fields}
For every $n\ge2$,
$
  \overline{\widetilde{\calL}_n}^{\,C^\infty}
  =
  \mathfrak X(S^n).
$
\end{proposition}

\begin{proof}
The divergence map has a continuous right inverse.  If $f\in
C^\infty_0(S^n)$, solve
$
  \Delta u=f
$
with zero mean and set
$
  \mathcal Rf=\nabla u.
$
Then
$
  \diver(\mathcal Rf)=f.
$

Let $Y$ be an arbitrary smooth vector field and put
$
  f=\diver Y.
$
Choose finite scalar spherical-harmonic sums $f_j$ such that
$
  f_j\longrightarrow f
$
in $C^\infty$.  By Proposition~\ref{prop:scalar-divergences},
$f_j\in\mathcal D_n$, so there is
$
  Z_j\in\widetilde{\calL}_n
$
with
$
  \diver Z_j=f_j.
$

Set
$
  Y_j
  =
  Y-\mathcal R(f-f_j).
$
Then
$
  Y_j\longrightarrow Y
$
in $C^\infty$, while
$
  \diver(Y_j-Z_j)=0.
$
The preceding four sections show that $\calL_n$ is $C^\infty$-dense in
the divergence-free vector fields.  We may therefore choose
$
  X_j\in\calL_n\subset\widetilde{\calL}_n
$
such that
$$
  X_j-(Y_j-Z_j)\longrightarrow0
$$
in $C^\infty$.  It follows that
$
  X_j+Z_j\longrightarrow Y
$
in $C^\infty$.  Hence
$
  \overline{\widetilde{\calL}_n}^{\,C^\infty}
  =
  \mathfrak X(S^n).
$
\end{proof}

Proposition~\ref{prop:all-vector-fields} is the infinitesimal statement
needed for Proposition~\ref{prop:nplusone}\textup{(a)}.  Together with the
divergence-free density proved in Sections~2--5, it leaves only the passage
from vector fields to their flows, which we carry out in the next section.

\section{From vector fields to diffeomorphisms}

We now complete the proof of Proposition~\ref{prop:nplusone}.  Let
$
  G(\mathcal E_n)
$
denote the subgroup generated by the flows of the fields in
$\mathcal E_n$, and let
$
  G(\widetilde{\mathcal E}_n)
$
denote the subgroup generated by the flows of the enlarged family
$\widetilde{\mathcal E}_n$.  Thus
$$
  G(\mathcal E_n)\subset\SDiff_0(S^n),
  \qquad
  G(\widetilde{\mathcal E}_n)\subset\Diff_0(S^n).
$$
The analytic argument is the same for the two families.  We first work in a
fixed finite $C^k$ topology and only at the end pass to $C^\infty$.

Fix an integer $k\ge1$.  We use the following standard facts about flows
of smooth vector fields on a compact manifold.  They are elementary
finite-regularity forms of standard product and commutator formulas for
flows, often organized under the heading of chronological calculus; see
\cite{AgrachevGamkrelidze,MauhartMichor}.  We include the estimates needed
here.

\begin{lemma}[Finite-$C^k$ flow calculus]\label{lem:flow}
Let $M$ be compact without boundary, and let $X,Y$ be smooth vector
fields on $M$.  Denote the flow of $X$ by $\Phi_X^t$.

\begin{enumerate}[label=(\roman*),leftmargin=2em]
\item The Lie--Trotter products
$$
  \left(\Phi_X^{t/N}\circ\Phi_Y^{t/N}\right)^N
$$
converge in $C^k$ to $\Phi_{X+Y}^t$, uniformly for $t$ in compact
intervals.

\item Define the group commutator
$$
  C_\varepsilon(X,Y)
  =
  \Phi_Y^{-\varepsilon}\circ
  \Phi_X^{-\varepsilon}\circ
  \Phi_Y^\varepsilon\circ
  \Phi_X^\varepsilon.
$$
Then
$$
  C_\varepsilon(X,Y)
  =
  \Phi_{[X,Y]}^{\varepsilon^2}
  +O_{C^k}(|\varepsilon|^3).
$$
Consequently, for $t>0$,
$$
  C_{\sqrt{t/N}}(X,Y)^N
  \longrightarrow
  \Phi_{[X,Y]}^t
$$
in $C^k$; negative $t$ follows by inversion.

\item If $X_j\to X$ in $C^{k+1}$, then
$
  \Phi_{X_j}^t\longrightarrow\Phi_X^t
$
in $C^k$, uniformly for $t$ in compact intervals.

\item The analogous statement holds on a fixed interval $[0,T]$ for
time-dependent vector fields that are smooth in space and piecewise
continuous in time with values in $C^{k+1}$.  If
$X_j(t,\cdot)\to X(t,\cdot)$ uniformly in $t$ in $C^{k+1}$, their
evolution maps, starting at the identity, converge uniformly in $t$ in
$C^k$.
\end{enumerate}
\end{lemma}

\begin{proof}
Because $M$ is compact and the vector fields are smooth, the maps
$h\mapsto\Phi_X^h$ and $h\mapsto\Phi_Y^h$ admit Taylor expansions at
$h=0$ with remainders controlled in $C^k$.  In local coordinates,
$$
  \Phi_X^h
  =
  \operatorname{id}+hX+O_{C^k}(h^2).
$$
It follows that
$$
  \Phi_X^h\circ\Phi_Y^h
  =
  \Phi_{X+Y}^h+O_{C^k}(h^2).
$$
With $h=t/N$, the local error is $O(N^{-2})$.  Uniform $C^{k+1}$
bounds for the short-time maps and the usual composition estimates give a
total error $O(N^{-1})$ after $N$ steps, proving (i).

For the commutator, Taylor expansion one order further gives
$$
  C_\varepsilon(X,Y)
  =
  \operatorname{id}
  +\varepsilon^2[X,Y]
  +O_{C^k}(|\varepsilon|^3),
$$
which is the classical commutator-of-flows formula.  Since
$$
  \Phi_{[X,Y]}^{\varepsilon^2}
  =
  \operatorname{id}
  +\varepsilon^2[X,Y]
  +O_{C^k}(\varepsilon^4),
$$
the one-step discrepancy is $O_{C^k}(|\varepsilon|^3)$.  The same Taylor
expansion is valid in $C^{k+1}$, and gives
$$
  C_\varepsilon(X,Y)
  =
  \operatorname{id}+O_{C^{k+1}}(\varepsilon^2).
$$
Consequently its iterates have uniformly bounded derivatives through order
$k$ while the number of iterates $m$ satisfies
$m\varepsilon^2\le t$.  Applying the composition estimates and discrete
Gronwall inequality to these iterates gives
$$
  \bigl\|C_\varepsilon(X,Y)^N
  -\Phi_{[X,Y]}^{N\varepsilon^2}\bigr\|_{C^k}
  \le C_tN|\varepsilon|^3.
$$
Taking $\varepsilon=\sqrt{t/N}$ gives an error $O(N^{-1/2})$, proving
(ii).

For (iii), write the integral equations for the flows.  Gronwall's
inequality gives convergence in $C^0$.  Differentiating with respect to
the initial point gives the first variational equation; higher derivatives
satisfy linear inhomogeneous variational equations involving only finitely
many derivatives of the vector fields and lower derivatives of the flow.
Induction through order $k$, followed by Gronwall's inequality at each
stage, proves convergence in $C^k$.  The same argument, with estimates
uniform in time, proves (iv).  Piecewise-continuous time dependence causes
no difficulty.
\end{proof}

We now apply Lemma~\ref{lem:flow} to the two generating families.

Let
$$
  \overline{G(\mathcal E_n)}^{\,k}
  \subset\Diff^k(S^n),
  \qquad
  \overline{G(\widetilde{\mathcal E}_n)}^{\,k}
  \subset\Diff^k(S^n)
$$
denote the corresponding $C^k$ closures.  Since $\Diff^k(S^n)$ is a
topological group, both closures are subgroups.

Let $\calA_{n,k}^{\,\mathrm{vol}}$ consist of the smooth divergence-free
vector fields $X$ for which
$$
  \Phi_X^t\in\overline{G(\mathcal E_n)}^{\,k}
  \qquad\text{for every }t\in\R,
$$
and let $\calA_{n,k}^{\,\mathrm{all}}$ consist of all smooth vector fields
$X$ for which
$$
  \Phi_X^t\in
  \overline{G(\widetilde{\mathcal E}_n)}^{\,k}
  \qquad\text{for every }t\in\R.
$$
Lemma~\ref{lem:flow} shows that both spaces are Lie algebras: time rescaling
gives scalar multiples, the Lie--Trotter formula gives sums, and the
commutator formula gives brackets.  Hence
$
  \calL_n\subset\calA_{n,k}^{\,\mathrm{vol}},
$ and
$  \widetilde{\calL}_n\subset\calA_{n,k}^{\,\mathrm{all}}.
$

Sections~2--5 show that $\calL_n$ is $C^\infty$-dense in the smooth
divergence-free vector fields, while
Proposition~\ref{prop:all-vector-fields} shows that
$\widetilde{\calL}_n$ is $C^\infty$-dense in all smooth vector fields.
If $X$ is a smooth vector field in the corresponding class, choose
$X_j$ in $\calL_n$, respectively in $\widetilde{\calL}_n$, with
$
  X_j\longrightarrow X
$
in $C^{k+1}$.  Lemma~\ref{lem:flow}(iii) then gives
$
  \Phi_{X_j}^t\longrightarrow\Phi_X^t
$
in $C^k$.  Since the appropriate closure is closed, we conclude that
$
  \calA_{n,k}^{\,\mathrm{vol}}
  =
  \mathfrak X_{\diver=0}(S^n)$
and
$  \calA_{n,k}^{\,\mathrm{all}}
  =
  \mathfrak X(S^n).
$

It remains to pass from autonomous vector fields to arbitrary isotopies.
The same time-slicing argument treats the two cases in parallel.

First let
$
  \Phi\in\Diff_0(S^n).
$
Choose a smooth isotopy $\Phi_t$, $0\le t\le1$, from the identity to
$\Phi$, and let
$
  Y_t=\dot\Phi_t\circ\Phi_t^{-1}
$
be its time-dependent generating vector field.  For a partition
$$
  0=t_0<t_1<\cdots<t_N=1,
$$
freeze $Y_t$ on each subinterval and form
$$
  \Psi_P
  =
  \Phi_{Y_{t_{N-1}}}^{t_N-t_{N-1}}
  \circ\cdots\circ
  \Phi_{Y_{t_0}}^{t_1-t_0}.
$$
Every factor lies in
$\overline{G(\widetilde{\mathcal E}_n)}^{\,k}$, and therefore so does
$\Psi_P$.  As the mesh tends to zero, Lemma~\ref{lem:flow}(iv) gives
$
  \Psi_P\longrightarrow\Phi
$
in $C^k$.  Hence
$
  \Diff_0(S^n)
  \subset
  \overline{G(\widetilde{\mathcal E}_n)}^{\,k}
$
for every finite $k\ge1$.  Equivalently,
$G(\widetilde{\mathcal E}_n)$ is $C^k$-dense in
$\Diff_0(S^n)$.

Now let
$
  \Phi\in\SDiff_0(S^n).
$
By the definition of $\SDiff_0(S^n)$, choose a smooth
volume-preserving isotopy $\Phi_t$ from the identity to $\Phi$.  Its
generating field $Y_t$ is divergence-free.  Repeating the same
time-slicing argument, now using
$\calA_{n,k}^{\,\mathrm{vol}}$, gives
$
  \SDiff_0(S^n)
  \subset
  \overline{G(\mathcal E_n)}^{\,k}
$
for every finite $k\ge1$.  Thus $G(\mathcal E_n)$ is $C^k$-dense in
$\SDiff_0(S^n)$.

Finally, let $\mathcal U$ be a $C^\infty$ neighborhood of a
diffeomorphism $\Phi$ in either target group.  By the definition of the
Fr\'echet topology, there are an integer $K\ge0$ and a $C^K$
neighborhood $\mathcal V$ of $\Phi$ such that the intersection of
$\mathcal V$ with the relevant target group is contained in
$\mathcal U$.  Taking $k\ge\max\{1,K\}$ and using the finite-$C^k$
density proved above yields an element of the corresponding generated group
in $\mathcal U$.  Since the generated groups themselves lie in the
corresponding identity components, we conclude that
$$
  \overline{G(\widetilde{\mathcal E}_n)}^{\,C^\infty}
  =
  \Diff_0(S^n),
  \qquad\mbox{and}\qquad
  \overline{G(\mathcal E_n)}^{\,C^\infty}
  =
  \SDiff_0(S^n).
$$
This proves both parts of Proposition~\ref{prop:nplusone}.

\section{From the natural generating families to two generators}
\label{sec:two-generators}

We now prove Proposition~\ref{prop:two-recover} and complete the proof of
Theorem~\ref{thm:main}.  Recall
$$
  A=R_{12},
  \qquad
  B=S+R_{23}+R_{34}+\cdots+R_{n,n+1},
  \qquad
  \widetilde B=B+D_n,
$$
where
$
  S=x_3R_{12},
$
and
$$
  D_n=\nabla_{S^n}h_n,
  \qquad
  h_n=
  \begin{cases}
    x_3, & n=2,\\[2mm]
    2x_3^3-3x_3(x_1^2+x_2^2), & n\ge3.
  \end{cases}
$$

Let
$
  \calM_n=\Lie\{A,B\}$
  and
$
  \widetilde{\calM}_n=\Lie\{A,\widetilde B\}.
$
We shall show that
$
  \mathcal E_n\subset\calM_n$
and
$  \widetilde{\mathcal E}_n\subset\widetilde{\calM}_n.
$

\subsection*{Recovering the first rotations}

The first brackets are the same in both cases.  Since $R_{12}$ commutes
with $S$, with $R_{k,k+1}$ for $k\ge3$, and with $D_n$, we have
$
  [A,B]
  =
  [A,\widetilde B]
  =
  [R_{12},R_{23}]
  =
  -R_{13}.
$
Hence
$
  R_{13}\in\calM_n\cap\widetilde{\calM}_n.
$
Bracketing once more with $A$,
$$
  [A,[A,B]]
  =
  [A,[A,\widetilde B]]
  =
  -[R_{12},R_{13}]
  =
  -R_{23}.
$$
Thus
$$
  R_{12},R_{13},R_{23}
  \in
  \calM_n\cap\widetilde{\calM}_n.
$$

These three rotations form the standard copy of
$\mathfrak{so}(3)$ acting in the first three coordinates.  Define its
Casimir operator on vector fields by
$$
  \mathcal C
  =
  \ad_{R_{12}}^2
  +
  \ad_{R_{13}}^2
  +
  \ad_{R_{23}}^2.
$$
Since the three rotations belong to both generated Lie algebras,
$\mathcal C$ preserves both $\calM_n$ and
$\widetilde{\calM}_n$.

We first record the action of $\mathcal C$ on the divergence-free terms
that occur in $B$.

\begin{lemma}\label{lem:casimir-recovery}
As a linear map on vector fields, $\mathcal C$ preserves the span of
$$
  \{S,R_{12},R_{23},\ldots,R_{n,n+1}\},
$$
and on the relevant fields satisfies
$$
  \mathcal C S=-6S,
  \qquad
  \mathcal C R_{12}=-2R_{12},
  \qquad
  \mathcal C R_{23}=-2R_{23},
  \qquad
  \mathcal C R_{34}=-2R_{34},
$$
while
$
  \mathcal C R_{k,k+1}=0$
  for
 $k\ge4$.

\end{lemma}

\begin{proof}
Since $[R_{12},S]=0$, only the other two terms of $\mathcal C$
contribute to $\mathcal C S$.  Direct calculation gives
$$
  [R_{13},S]
  =
  -x_1R_{12}-x_3R_{23}
\qquad
\mbox{and}\qquad
  [R_{23},S]
  =
  -x_2R_{12}+x_3R_{13}.
$$
Bracketing once more,
$$
  [R_{13},[R_{13},S]]
  =
  -2S+2x_1R_{23}
  \qquad\mbox{and}\qquad
  [R_{23},[R_{23},S]]
  =
  -2S-2x_2R_{13}.
$$
But
$
  x_1R_{23}-x_2R_{13}=-S,
$
and hence
$
  \mathcal C S=-6S.
$

The fields $R_{12},R_{13},R_{23}$ form the adjoint representation of
$\mathfrak{so}(3)$, on which the Casimir acts by $-2$; in particular
$
  \mathcal C R_{12}=-2R_{12}$
  and
  $
  \mathcal C R_{23}=-2R_{23}.
$

For $R_{34}$,
we have $
  [R_{13},R_{34}]=-R_{14}$
and
$ [R_{13},R_{14}]=R_{34},
$
so
$
  \ad_{R_{13}}^2R_{34}=-R_{34}.
$
Similarly,
$
  \ad_{R_{23}}^2R_{34}=-R_{34},
$
while $R_{12}$ commutes with $R_{34}$.  Thus
$
  \mathcal C R_{34}=-2R_{34}.
$

Finally, if $k\ge4$, the indices of $R_{k,k+1}$ are disjoint from
$1,2,3$, so it commutes with all three rotations defining
$\mathcal C$.  Hence
$
  \mathcal C R_{k,k+1}=0.
$
\end{proof}

The gradient terms are just as simple.  For every rotation field $R$,
equivariance of the gradient gives
$
  [R,\nabla f]=\nabla(Rf).
$
Consequently
$$
  \mathcal C D_n
  =
  \nabla\!\left(
  (R_{12}^2+R_{13}^2+R_{23}^2)h_n
  \right).
$$
For $n=2$, $h_2=x_3$ is a degree-one spherical harmonic for this
$SO(3)$-action, and hence
$
  \mathcal C D_2=-2D_2.
$
For $n\ge3$, $h_n$ is a degree-three spherical harmonic in the first
three coordinates, and therefore
$
  \mathcal C D_n=-12D_n.
$

\subsection*{The full diffeomorphism family}

We first prove Proposition~\ref{prop:two-recover}\textup{(a)}.

Suppose first that $n=2$.  Set
$
  \widetilde Z
  =
  \widetilde B-R_{23}
  =
  D_2+S.
$
The two terms have distinct Casimir eigenvalues:
$
  \mathcal C D_2=-2D_2,
$ and
$\mathcal C S=-6S.
$
Hence
$$
  D_2
  =
  \frac14(\mathcal C+6)\widetilde Z
  \qquad
\mbox{and}
\qquad
  S
  =
  -\frac14(\mathcal C+2)\widetilde Z.
$$
Thus
$
  D_2,S\in\widetilde{\calM}_2.
$
Together with  $R_{12}$ and $R_{23}$, these are
precisely the fields in $\widetilde{\mathcal E}_2$.

Now suppose $n\ge3$.  Again set
$
  \widetilde Z
  =
  \widetilde B-R_{23}.
$
Then
$$
  \widetilde Z
  =
  D_n+S+R_{34}+\Bigl(R_{45}+\cdots+R_{n,n+1}\Bigr).
$$
The four types of terms occurring here have Casimir eigenvalues
$
  -12$, $-6$, $-2$, and 0,
respectively.  Polynomial filtering therefore gives
$$
  D_n
  =
  -\frac1{720}
  \mathcal C(\mathcal C+2)(\mathcal C+6)\widetilde Z,
\qquad
\mbox{and}
\qquad
  S
  =
  \frac1{144}
  \mathcal C(\mathcal C+2)(\mathcal C+12)\widetilde Z.
$$
Hence
$
  D_n,S\in\widetilde{\calM}_n.
$

Subtracting the recovered fields from $\widetilde Z$ gives
$$
  T_3
  =
  R_{34}+R_{45}+\cdots+R_{n,n+1}
  \in
  \widetilde{\calM}_n.
$$
The remaining rotations can now be recovered successively.

Suppose, for some $j$, $3\le j\le n$, that
$
  R_{j-1,j}\in\widetilde{\calM}_n
$
and
$$
  T_j
  =
  R_{j,j+1}+R_{j+1,j+2}+\cdots+R_{n,n+1}
  \in
  \widetilde{\calM}_n.
$$
All terms of $T_j$ except the first commute with $R_{j-1,j}$, so
$
  [R_{j-1,j},T_j]
  =
  -R_{j-1,j+1}.
$
Bracketing once more gives
$
  [R_{j-1,j},[R_{j-1,j},T_j]]
  =
  -R_{j,j+1}.
$

Thus
$
  R_{j,j+1}\in\widetilde{\calM}_n,
$
and subtraction yields
$$
  T_{j+1}
  =
  T_j-R_{j,j+1}
  \in
  \widetilde{\calM}_n.
$$
Starting with the already known $R_{23}$ and $T_3$, induction recovers
$
  R_{34},R_{45},\ldots,R_{n,n+1}.
$
Therefore
$
  \widetilde{\mathcal E}_n
  \subset
  \widetilde{\calM}_n,
$
which proves Proposition~\ref{prop:two-recover}\textup{(a)}.

\subsection*{The volume-preserving family}

The proof of Proposition~\ref{prop:two-recover}\textup{(b)} is the same
argument with the gradient component omitted.

Set
$
  Z=B-R_{23}.
$
For $n=2$,
$
  Z=S,
$
so the recovery is immediate.

For $n\ge3$,
$$
  Z
  =
  S+R_{34}+\Bigl(R_{45}+\cdots+R_{n,n+1}\Bigr).
$$
The three types of terms have Casimir eigenvalues
$ -6$, $-2$ and 0.
Hence
$$
  S
  =
  \frac1{24}
  \mathcal C(\mathcal C+2)Z
  \in
  \calM_n.
$$
Subtracting $S$ gives
$$
  T_3
  =
  R_{34}+R_{45}+\cdots+R_{n,n+1}
  \in
  \calM_n.
$$
The same tail-recovery induction used above then gives
$$
  R_{34},R_{45},\ldots,R_{n,n+1}
  \in
  \calM_n.
$$
Together with $R_{12}$, $R_{23}$, and $S$, these are precisely the
fields in $\mathcal E_n$.  Thus
$
  \mathcal E_n\subset\calM_n,
$
proving Proposition~\ref{prop:two-recover}\textup{(b)}.

\subsection*{Proof of the main theorem}

We now combine Propositions~\ref{prop:nplusone} and
\ref{prop:two-recover}.

Fix $k\ge1$, and let
$
  \overline{G(A,\widetilde B)}^{\,k}
 $
 and 
$
  \overline{G(A,B)}^{\,k}
$
denote the corresponding closures in $\Diff^k(S^n)$.  By
Lemma~\ref{lem:flow}, the smooth vector fields whose flows lie in either
closure for every time form a Lie algebra.

Proposition~\ref{prop:two-recover}\textup{(a)} therefore implies that the
flow of every field in $\widetilde{\mathcal E}_n$ lies in
$\overline{G(A,\widetilde B)}^{\,k}$.  Hence
$
  G(\widetilde{\mathcal E}_n)
  \subset
  \overline{G(A,\widetilde B)}^{\,k}.
$
By Proposition~\ref{prop:nplusone}\textup{(a)},
$G(\widetilde{\mathcal E}_n)$ is $C^\infty$-dense in
$\Diff_0(S^n)$, and therefore $C^k$-dense there.  It follows that
$
  \Diff_0(S^n)
  \subset
  \overline{G(A,\widetilde B)}^{\,k}
$ for every finite $k\ge 1$.
Equivalently, $G(A,\widetilde B)$ is $C^k$-dense in
$\Diff_0(S^n)$.

Similarly, Proposition~\ref{prop:two-recover}\textup{(b)} implies
$
  G(\mathcal E_n)
  \subset
  \overline{G(A,B)}^{\,k},
$
and Proposition~\ref{prop:nplusone}\textup{(b)} gives
$
  \SDiff_0(S^n)
  \subset
  \overline{G(A,B)}^{\,k}.
$
Thus $G(A,B)$ is $C^k$-dense in $\SDiff_0(S^n)$.

Finally, let $\mathcal U$ be a $C^\infty$ neighborhood of an arbitrary
diffeomorphism $\Phi$ in either target group.  Choose a finite $K$ and a
$C^K$ neighborhood $\mathcal V$ of $\Phi$ such that the intersection
of $\mathcal V$ with the relevant target group is contained in
$\mathcal U$.  Taking $k\ge\max\{1,K\}$ and using the finite-$C^k$
density just proved gives an element of the corresponding two-generator
group in $\mathcal U$.  Since
$
  G(A,\widetilde B)\subset\Diff_0(S^n)
  $
and
$  G(A,B)\subset\SDiff_0(S^n),
$
we conclude that
$
  \overline{G(A,\widetilde B)}^{\,C^\infty}
  =
  \Diff_0(S^n)$
and
$  \overline{G(A,B)}^{\,C^\infty}
  =
  \SDiff_0(S^n).
$

This proves both parts of Theorem~\ref{thm:main}.

\section{Related work}

Agrachev and Sarychev developed Lie-algebraic approximation methods for
diffeomorphism groups and applied them to controllability questions
\cite{AgrachevSarychev2008,AgrachevSarychev2022}.  On $S^2$ their
Hamiltonian generating family uses three independent linear spherical
harmonics, one quadratic harmonic, and one cubic harmonic.  In the present
formulation two nonparallel linear Hamiltonians generate the third by Poisson
bracket, while the quadratic Hamiltonian $x_3^2/2$ already supplies the
degree-raising mechanism.  Thus no separate cubic Hamiltonian is needed.
The same three-dimensional ladder persists inside the transverse vector
spherical harmonics on every $S^n$.

The recent work of Jiang and Chen \cite{JiangChen2026} is more directly
related to the two-generation result for the full diffeomorphism group.
They prove that the Lie algebra of algebraic vector fields on $S^2$ is
generated exactly by two explicit fields and deduce approximation of every
diffeomorphism isotopic to the identity.  Thus they have priority for the
two-generator theorem for $\Diff_0(S^2)$, and their result is stronger at
the algebraic level.  Their construction and the $S^2$ case here share the
same basic features: one generator is a rotation, the other combines
Hamiltonian and gradient pieces, and the components are separated
spectrally.  Jiang and Chen also announce a forthcoming higher-dimensional
extension of their algebraic result.

Berger and Turaev \cite{BergerTuraev} study generation by nonlinear shears
in analytic Hamiltonian dynamics in flat symplectic settings.  Related
algebraic generation questions have also been considered in affine space;
in particular, Beldiev and Pogudin \cite{BeldievPogudin2026} obtain a
$1.5$-generation result for polynomial vector fields of constant
divergence.  These works provide a broader context for the finite-generation
questions considered here.

\subsection*{Acknowledgements}

I am grateful to Jerry Kazdan for first calling the generalized Minkowski problem to my attention, and to 
Herman Gluck for asking, in the course of discussions about that problem, what 
subgroup of $\SDiff_0(S^2)$ --- and, more generally, of $\SDiff_0(S^n)$ --- might be generated by 
rotational shears. That question led to the present paper.


\end{document}